\documentclass[12pt, reqno]{amsproc}
\usepackage{amsmath,amssymb,mathrsfs,hyperref,color}

\DeclareMathAlphabet{\mathpzc}{OT1}{pzc}{m}{it}
\usepackage[nameinlink]{cleveref}
\hypersetup{colorlinks={true},linkcolor={blue},citecolor=violet}
\usepackage{mathtools}
\usepackage[x11names]{xcolor}
\usepackage{paralist}

\usepackage{subfigure}
\usepackage{float}
\usepackage{times}
\usepackage{tikz}
\usepackage{cases}
\usetikzlibrary{intersections}
\usepackage{xcolor}
\usepackage[latin1]{inputenc}
\usepackage{bbm}
\usepackage{tabularx}
\usepackage{booktabs}
\usepackage[labelfont=bf,format=plain,justification=raggedright,singlelinecheck=false]{caption}
\usetikzlibrary{shadings}
\usetikzlibrary[intersections]
\usetikzlibrary{arrows,shapes}

\makeatletter
\def\setliststart#1{\setcounter{\@listctr}{#1}%
  \addtocounter{\@listctr}{-1}}
\makeatother

\makeatletter
\@addtoreset{figure}{section}
\makeatother

\usepackage[mathcal]{euscript}
\usepackage{calc}
\def\R{\mathbb R}

\DeclareMathOperator*{\eps}{\varepsilon}

\newtheorem{theorem}{{\bf Theorem}}[section]
\newtheorem{lemma}[theorem]{Lemma}
\newtheorem{proposition}[theorem]{Proposition}

\newtheorem{assumption}[theorem]{Assumption}

\numberwithin{equation}{section}

\title[Periodic limit for non-autonomous Lagrangian]{Periodic limit for non-autonomous Lagrangian systems and applications to a Kuramoto type model}
\author{Veronica Danesi \and Cristian Mendico \and Xuan Tao \and Kaizhi Wang}
\address{Dipartimento di matematica, Universit\'a degli studi di Roma Tor Vergata -- Via della Ricerca Scientifica 1, 00133 Roma}
\email{danesi@mat.uniroma2.it}
\address{Institut de Math\'ematique de Bourgogne, UMR 5584 CNRS, Universit\'e Bourgogne Europe, 21000 Dijon, France}
\email{cristian.mendico@u-bourgogne.fr}
\address{Huatai Securities Company Limited -- Nanjing 210019, China}
\email{taoxuan@htsc.com} 
\address{School of Mathematical Sciences, Shanghai Jiao Tong University -- Shanghai 200240, China}
\email{kzwang@sjtu.edu.cn}

\date{\today}
\thanks{{\it Acknowledgements:} Veronica Danesi was partially supported by the MUR Excellence Department Project awarded to the Department of Mathematics, University of Rome Tor Vergata, CUP E83C23000330006, by the PRIN 2022 PNRR-Project P20225SP98 \textquotedblleft Some mathematical approaches to climate change and its impacts\textquotedblright (funded by the European Community-Next Generation EU, CUP E53D2301791 0001) and by the PRIN Project 2022FPZEES \textquotedblleft Stability in Hamiltonian Dynamics and
Beyond\textquotedblright. The second author wishes to thank the School of Mathematical Sciences for the hospitality at Shanghai Jiao Tong University (China) during which this paper was finished. Kaizhi Wang is partially supported by National Natural Science Foundation of China (Grant
Nos. 12525107, 12171315).}

\makeindex

\begin{document}

%%%%%%%%%%%%%%%%%%%%%%%%%%%%%%%%%%%%%%%%%%%%%%%%%%%%%%%%%%%%%%%%%%

\maketitle

\begin{abstract}
This paper explores the asymptotic properties of non-autonomous Lagrangian systems, assuming that the associated Tonelli Lagrangian converges to a time-periodic function. Specifically, given a continuous initial condition, we provide a suitable construction of a Lax-Oleinik semigroup such that it converges toward a periodic solution of the equation. Moreover, the graph of its gradient converges as time tends to infinity to the graph of the gradient of the periodic limit function with respect to the Hausdorff distance. Finally, we apply this result to a Kuramoto-type model, proving the existence of an invariant torus given by the graph of the gradient of the limiting periodic solution of the Hamilton-Jacobi equation.
\end{abstract}

%\tableofcontents

\section{Introduction}

The goal of this paper is to describe the asymptotic properties of non-autonomous Lagrangian systems where the Lagrangian function converges as time tends to infinity toward a periodic Tonelli Lagrangian. This will be achieved by constructing a suitable Lax-Oleinik semigroup such that it converges uniformly to the periodic viscosity solution of the periodic limit equation for any continuous initial datum. Furthermore, the graph of its gradient also converges with respect to the Hausdorff distance.

In the classical Tonelli case for autonomous systems such results are well established. Indeed, given a closed manifold $M$, endowed with a Riemannian metric, let $L:TM\to\R$ be a Tonelli Lagrangian and we denote by $H:T^*M\to \R$ its associated Hamiltonian. The corresponding stationary Hamilton-Jacobi equation is
\begin{equation}\label{1-2}
H(x,d u)=c(L),
\end{equation}
where $c(L)$ is the Ma$\mathrm{\tilde{n}}\mathrm{\acute{e}}$ critical value of $L$ and, without loss of generality, we will from now on always assume $c(L)=0$. For each $u\in C(M,\R)$, each $t\geq0$, and each $x\in M$ we recall that the Lax-Oleinik semigroup is defined as
\begin{equation}\label{1-3}
T_t u(x)=\inf_\gamma \Big\{u(\gamma(0))+\int_0^t L(\gamma(s),\dot{\gamma}(s))ds\Big\}
\end{equation}
where the infimum is taken among the continuous and piecewise $C^1$ paths $\gamma:[0,t]\to M$ with $\gamma(t)=x$. In \cite{Fat1} it has been proved the existence of the weak KAM solutions of the stationary Hamilton-Jacobi equation (\ref{1-2}) by showing the existence of fixed points of the Lax-Oleinik semigroup. 
Furthermore,  it is known that  for any $u \in C(M, \R)$ we have that $\displaystyle{\lim_{t\to+\infty}}T_tu=\bar{u}$ exists and $\bar u$ is a weak KAM solution of (\ref{1-2}), see \cite{Fat4}. Finally,  a geometric interpretation of the convergence of the Lax-Oleinik semigroup has been provided in \cite{Arn}: the family of the adherences of the graphs of $dT_tu$ converges for the topology of Hausdorff to the adherence of the graph of $d\bar{u}$ as $t\to+\infty$.
For more on weak KAM theory we refer to \cite{Andrea_1, Sorrentino_1, Maxime_1, WL, W} and the references therein.

For the purpose of this paper, let $L_1: TM \times \R \to \R$  be a non-autonomous Tonelli Lagrangian, see Assumption \ref{TonelliAssumption} below. In this paper, we address the asymptotic behavior of a modified Lax-Oleinik semigroup
\[
\mathcal{T}_t \varphi(x) = T^1_t \varphi(x) - \inf_{x \in M} T^1_t \varphi(x), \quad (t > 0,\;\; \varphi \in C(M;\R)),
\]
%\begin{equation}\label{intro_Cauchy}
%    \begin{cases}
%    \partial_t u(t, x) + H(x, du(t, x), t) = 0, & (t, x) \in (0, \infty) \times M
%    \\
%    u(0,x ) = u_0(x), & u_0 \in C(M, \R),
%    \end{cases}
%\end{equation}
where $T^1_t\varphi$ denotes the classical Lax-Oleinik operator associated with $L_1$, under the assumption that $L_1$ converges to a 1-periodic Tonelli Lagrangian function $\overline{L_1}:TM \times \mathbf{S}^1 \to \R$, i.e., 
\begin{equation}\label{Lagrangian_convergence}
\displaystyle{\lim_{n \to \infty}} L_1(x, v, t+n) = \overline{L_1}(x, v, t) \quad \textrm{in}\;\; C^2_c(TM \times \R;\R)
\end{equation}
where $C^2_c(TM \times \R;\R)$ denotes the space of $C^2$ functions with compact support in $TM \times \R$, with exponential rate of convergence
\begin{equation}\label{rate_Lagrangian}
\| L_1(t + n, x, v) - \overline{L_1}(t, x, v) \|_{C^2_c( T M\times \R;\R)} \leq C e^{-\rho n}, \mbox{ for some } C\in\mathbb{R},\;\rho > 0\; \mbox{ and }\; \forall\; n \in \mathbf{N}.
\end{equation} 
Moreover, we assume that the Aubry set of $\overline{L_1}$ consists of a unique hyperbolic periodic orbit.
%Observe that, if the Lagrangian $\overline{L}$ is $T$-periodic, for some $T > 0$, assumption \eqref{rate_Lagrangian} reads as follows \begin{equation*}
%\| L(t + n, x, v) - \overline{L_1}(t, x, v) \|_{C^2_c( T M\times \R;\R)} \leq C e^{-\rho T n}, \mbox{ for some } C\in\mathbb{R} \mbox{ and }\rho > 0.
%\end{equation*}

Finally, we denote by $H_1: T^*M \times \R \to \R$ the Hamiltonian associated with $L_1$:
\[
H_1(t, x, p) = \sup_{v \in T_x M} \{\langle p, v \rangle - L_1(t, x, v) \}
\]
and similarly for $\overline{H_1}$.

\medskip

The main result of this paper is the following. 

\begin{theorem}\label{thm_2_intro}
%Let $z: [0, \infty) \times M \to \R$ be a solutions of \eqref{tilde}. Then, the following hold.
Let $\varphi \in C(M;\R)$. Then, the following hold.
\begin{itemize}
\item[($i$)] For any $\varphi \in C(M;\R)$ there exists a time periodic viscosity solution $w \in C(M \times \R; \R)$ to 
\begin{equation*}
    \partial_t w(x, t) + \overline{H_1}(t, x, d_x w(x,t)) = 0
\end{equation*}
such that 
\[
\displaystyle{\lim_{n \to \infty}} \| \mathcal{T}_{t+n} \varphi(\cdot) -  w(\cdot, [ t ])\|_{\infty} = 0
\]
for all $t \in \R$, where $[t] = t\; \mbox{mod } 1$. Moreover, 
\begin{equation*}
w(x, [ t ]) = \overline{\varphi}(x, [ t ]) - \inf_{x \in M} \overline{\varphi}(x, [ t ])
\end{equation*}
where
\begin{equation*}
    \bar{\varphi}(x,[t])=\inf_{y\in M}\big(\varphi(y)+h_{0,[t]}(y,x)\big),
\end{equation*} 
where $h_{0,[t]}(y,x)$ denotes the extended Peierls barrier defined below in \eqref{Ext_Peierls},  and for each $\varphi \in C(M; \R)$ we have 
\begin{equation}\label{exp_1}
\| \mathcal{T}_{t+n} \varphi(\cdot) - w(\cdot, [ t ]) \|_{\infty} \leq C e^{-\rho n}, \;\; \forall\; n \in \mathbf{N}, \; \forall\; t \in \R \mbox{ and for some } C\in\mathbb{R}
\end{equation}
where $\rho$ is given in \eqref{rate_Lagrangian}.
\item[($ii$)]  For any $\varphi \in C(M;\R)$ we have $\displaystyle{\lim_{n \to \infty}} d_{H} (\overline G_n(d\mathcal{T}\varphi), \overline G(dw)) = 0$ where 
\begin{equation*}
G_n(d\mathcal{T}\varphi):=\Big\{\big(x,[t],d_x \mathcal{T}_{n+[t]} \varphi(x),d_{t}\mathcal{T}_{n+[t]} \varphi(x)\big): (x, n+ [t]) \in \mbox{Dom}(d\mathcal{T}\varphi) \Big\}
\end{equation*}
and $G(dw)$ is the adherence of the limit function $w$. 
\end{itemize} 
\end{theorem}

%In particular, we show that $\mathcal{T}_t \varphi(x)$ converges uniformly to a $1-$periodic viscosity solution $\overline{u}: \mathbf{S}^1 \times M \to \R$ to 
%\begin{equation*}
%    \partial_t u(t, x) + \overline{H}(x, du(t, x), t)= 0
%\end{equation*}
%and, moreover, 
%\begin{equation}\label{intro_Hausdorff}
%    \lim_{n \to \infty} d_H\left(\overline{\textrm{graph}(d\mathcal{T}_{n+t} \varphi(x))}, \overline{\textrm{graph}(d\overline{u}(t, \cdot))} \right) = 0 
%\end{equation}
%where $d_H$ denotes the Hausdorff distance.

%In order to achieve such a result we need first to show that \eqref{intro_Hausdorff} holds for time-periodic Hamiltonian systems (see \Cref{thm_1} below).

Heuristically, the idea of the proof of such a result is to connect the function $\mathcal{T}_{t+n} \varphi$,associated with the non-periodic Lagrangian, to the new Lax-Oleinik (see \cite{WY1,WY2}) associated with the periodic Lagrangian. For more on weak KAM theory for time-periodic Hamiltonian systems we refer to \cite{CIS,FM,WY1,WY2} and to \Cref{weakKAM} below. 

We observe that for a time-periodic Lagrangian the convergence in \eqref{exp_1} was proved in \cite{W1}. More precisely, given $L_2: TM \times \mathbf{S}^{1} \to \R$, and $H_2: T^* M \times \mathbf{S}^{1} \to \R$ the corresponding Hamiltonian, the author showed that the viscosity solution of the Cauchy problem 
\begin{equation*}
    \begin{cases}
     \partial_t U(t, x) + H_2(t, x, d_xU(t, x)) = 0, & (t, x) \in (0, \infty) \times M
     \\
     U(0, x) = u_0(x), & u_0 \in C(M;\R)
    \end{cases}
\end{equation*}
converges to a periodic viscosity solution of 
\begin{equation*}
    \partial_t U(t, x) + H_2(t, x, d_xU(t, x)) = 0
\end{equation*}
with exponential rate of convergence under the assumption that the Aubry set is given by a unique hyperbolic periodic orbit. However, in this manuscript the Lagrangian $L$ is not periodic as stated in assumption \eqref{Lagrangian_convergence}. Furthermore, we observe that the exponential rate of convergence in \eqref{exp_1} is the most natural since, due to the hyperbolicity of the Aubry set, the minimizers of the initial non-periodic problem are attracted with exponential rate to the minimizers of the periodic limit problem.

\medskip

In conclusion, we consider the specific case of a Kuramoto type system, i.e., $N$ coupled oscillators described by the following equations
\begin{equation*}
\ddot{\theta}_i=\Omega_i+\sum_{j=1}^N a_{ij}(t)\sin(\theta_j-\theta_i)\ ,
\end{equation*}
where $\theta_i$ and $\Omega_i\in\mathbb{R}$ are the phase and natural frequency of i-th oscillator, respectively, and the co-efficients $a_{ij}$ represent the coupling between the j-th oscillator and the i-th oscillator. 
As a consequence of our main result \Cref{thm_2_intro} we get the existence of an invariant torus supported on the graph of the gradient of the limit time-periodic function (see \Cref{thm_3}). 
The Kuramoto model, in both its first and second-order forms, has found wide-ranging applications in physics, biology, neuroscience, and engineering. 
It has been used to study several synchronization phenomena in neuronal behavior, cardiac pacemaker cells, and the collective dynamics of power systems (see, for instance, \cite{Choi-19,Cumin-07, Strogatz-00, Witthaut-14} and references therein).
However, in many of these real-world systems, the coupling strength is not constant but rather varies in time due to external influences or internal adaptive processes. 
For example, neural connectivity can fluctuate and electrical loads in power grids vary over time. 
In \cite{Cumin-07}, time-varying coupling strengths and natural frequencies have been taken into account to provide more realistic pictures of neuronal synchronization in the brain. 
Similarly, other studies have included time-varying parameters, delayed couplings, or periodically forced versions of Kuramoto models (see, for instance, \cite{Leander-15,Lu-18, Metivier-20, Petroski-12, Wright-21} and references therein).
Since the inclusion of time-dependent coupling arises naturally in this kind of models (see also \cite{Niu-Wang-Li}), in this work we apply our weak KAM result to a modified second-order Kuramoto model with time-dependent coupling.

\subsection*{Organization of the paper}
\Cref{weakKAM} is dedicated to the review of the main definitions and results on weak KAM theory for autonomous and non-autonomous Lagrangian systems. \Cref{sec_thm1} and \Cref{sec_thm2} are devoted to the proof of the main results through several preliminary results having their own interests. Finally, in \Cref{Kuramoto} we address the application to a Kuramoto type model.

\section{On weak KAM theory}\label{weakKAM}

Hereafter, $M$ denotes a compact and connected smooth manifold without boundary endowed with a Riemannian metric,
and $TM$ and $T^*M$ are its tangent and cotangent bundles. 

\begin{assumption}\label{TonelliAssumption}
Let $L: TM\times\R\to\R$, $(x,v,t)\mapsto L(x,v,t)$ be of class $C^\infty$ and satisfy
\begin{itemize}
  \item[($i$)] {\bf convexity}: for all $x\in M$ and $t\in \R$, the Hessian matrix $\big(\partial^2 L/\partial v_i\partial v_j\big)(x,v,t)$ (calculated with respect to linear coordinates on $T_xM$) is positive definite;
  \item[($ii$)] {\bf superlinearity}: $\displaystyle{\lim_{\|v\|_x\to+\infty}}\frac{L(x,v,t)}{\|v\|_x}=+\infty$ uniformly on $x\in M$, $t\in\R$;
  \item[($iii$)] {\bf completeness}: all the maximal solutions of the Euler-Lagrange equation of $L$ are defined on $\R$.
\end{itemize}
Such a Lagrangian function will be called a {\bf Tonelli Lagrangian function}. 
\end{assumption}

We can associate with $L$ a Hamiltonian, as a function on
$T^*M\times\R$: 
\[
H(x,p,t)=\sup_{v\in T_xM}\{\langle p,v\rangle_x-L(x,v,t)\},
\]
where $\langle\cdot,\cdot\rangle_x$ represents the canonical pairing between the tangent and cotangent space. The corresponding evolutionary Hamilton-Jacobi
equation is

\begin{equation}\label{1-1}
d_tw+H(x,d_xw,t)=c(L),
\end{equation}
where $c(L)$ is the Ma$\mathrm{\tilde{n}}\mathrm{\acute{e}}$ critical value  of $L$ \cite{Man97}.

We have to recall the fundamental constructions of the weak KAM theory before we can state our
main result.  See \cite{Arn,Fat1,Fat2,Fat3,Fat4,Fat-b} and \cite{CIS,FM,WY1,WY2} for more details.

%\subsection{Weak KAM theory for autonomous Lagrangians}

\subsection{Weak KAM theory for time-periodic Lagrangians}

 In this section we introduce the notation used in the sequel and review some definitions and results of the weak KAM theory.

Let $A$ be a  subset of a metric space $(X,d)$. For $\varepsilon>0$, the ball of radius $\varepsilon>0$ around $A$ in $X$ is denoted by

\[
A^\varepsilon:=\{x\in X: d(x,A)\leq \varepsilon\}.
\]
We view $\mathbf{S}^1$ as a fundamental domain in $\R$, i.e., $[0,1]$ with the two endpoints identified. The standard universal covering projection $\pi:\R\to\mathbf{S}^1$ takes the form $\pi(t)=[t]$, where $[t]=t$ mod 1, denotes the fractional part of $t$, i.e., $t=[t]+\{t\}$, where $\{t\}$ is the greatest integer not greater than $t$.

Given a Tonelli Lagrangian $L: TM \times \R \to \R$ as in \eqref{TonelliAssumption}, the Euler-Lagrange equation generates a flow of diffeomorphisms $\phi^L_t:TM\times\mathbf{S}^1\to TM\times\mathbf{S}^1$, $t\in\R$, defined by

\[
\phi^L_t(x_0,v_0,t_0)=(x(t+t_0),\dot{x}(t+t_0),[ t+t_0]),
\]
where $x:\R\to M$ is the maximal solution of the
Euler-Lagrange equation with initial conditions $x(t_0)=x_0$,
$\dot{x}(t_0)=v_0$. The completeness and periodicity conditions
grant that this correctly defines a flow on $TM\times\mathbf{S}^1$.

For each $t\geq 0$ and each $u\in C(M,\R)$, let

\begin{equation*}
T_tu(x):=\inf_\gamma\Big\{u(\gamma(0))+\int_0^tL(\gamma,\dot{\gamma},s)ds\Big\}
\end{equation*}
for all $x\in M$, where the infimum is taken among the continuous and piecewise $C^1$ paths $\gamma:[0,t]\to M$ with $\gamma(t)=x$.
For each $t\geq 0$, $T_t$ is an operator from $C(M,\R)$ to itself. Since $L$ is time-periodic, then
$\{T_n\}_{n\in\mathbf{N}}$ is a one-parameter semigroup of operators, called the Lax-Oleinik semigroup associated with $L$, where
$\mathbf{N}=\{0,1,2,\dots\}$.  In \cite{Fat4} Fathi raised the question as to whether the convergence result of the Lax-Oleinik semigroup holds in the time-periodic case. This would be the convergence of $T_nu$, for all $u\in C(M,\R)$, as $n\to+\infty$, $n\in\mathbf{N}$. Later Fathi and Mather \cite{FM} provided examples with $M=\mathbf{S}^1$ where there is no such convergence, thus answering the above question negatively.

 Wang and Yan \cite{WY1}  introduced a suitable notion of Lax-Oleinik type operators associated with $L$ that reads as: for each $\tau\in[0,1]$, each $n\in\mathbf{N}$ and each $u\in C(M,\R)$, let

\[
\widetilde{T}_n^\tau u(x)= \inf_{\substack{k\in\mathbf{N} \\ n\leq k\leq
2n}}\inf_{\gamma}\Big\{u(\gamma(0))+\int_{0}^{\tau+k}
L(\gamma,\dot{\gamma},s)ds\Big\}
\]
for all $x\in M$, where the second infimum is taken among the continuous and piecewise $C^1$ paths $\gamma:[0,\tau+k]\rightarrow
M$ with $\gamma(\tau+k)=x$.
They also proved the convergence of the family of the new operators. 
%Now recall the definition of the new operator. 
 For each $\tau\in[0,1]$ and each $n\in \mathbf{N}$, $\widetilde{T}_n^\tau$ is an operator from $C(M,\R)$ to
itself and we call $\widetilde{T}_n^\tau$ the new Lax-Oleinik operator associated with $L$. Next, for each $n\in\mathbf{N}$ and each $u\in C(M,\R)$, let
\[
U^u_n(x,\tau)=\widetilde{T}_n^\tau u(x)
\]
for all $(x,\tau)\in M\times[0,1]$. In \cite{WY1} the authors proved the following results:
for each $u\in C(M,\R)$, the uniform limit 
\[
\lim_{n\to+\infty}U^u_n=\bar{u}
\]
exists, $\bar{u}$ is a weak KAM solution of the evolutionary  Hamilton-Jacobi equation (\ref{1-1}), and moreover 
\begin{equation}\label{u_bar}
\bar{u}(x,[\tau])=\inf_{y\in M}\big(u(y)+h_{0,[\tau]}(y,x)\big)
\end{equation}
for all $(x,\tau)\in M\times[0,1]$, where $h:\mathbf{S}^1\times\mathbf{S}^1\times M\times M\to\R$ denotes the extended Peierls barrier \cite{Mat93}, $[\tau]=\tau$ mod 1 and $\mathbf{S}^1=\R/\mathbf{Z}$. 

Finally, let $w\in C(M\times\mathbf{S}^1,\R)$. Then $w$ is a weak KAM solution of (\ref{1-1}) if and only if it satisfies
\[
\widetilde{T}^{\tau}_nw(x,0)=w(x,[\tau])
\] 
for all $n\in\mathbf{N}$ and for all $ (x,\tau)\in M\times[0,1]$. We recall that in the time-periodic case, weak KAM solutions and 1-periodic viscosity solutions are  the same.

\subsubsection{More on the Lax-Oleinik semigroup}
Under the assumptions \eqref{TonelliAssumption} on $L$, the Cauchy Problem for (\ref{1-1}) is well posed in the viscosity
sense: given a continuous function $u:M\to\R$, (\ref{1-1}) admits a unique continuous viscosity solution $U:M\times[0,+\infty)\to\R$ defined by $U(x,t)=T_tu(x)$ which is locally Lipschitz on $M\times(0,+\infty)$. 
%$w(x,t):M\times[0,+\infty)\to\R$, such that $w(\cdot,0)=u$. It is well known that for each $u\in C(M,\R)$, the function  is continuous on $M\times[0,+\infty)$ and . Furthermore, $U(x,t)$ is the unique viscosity solution of the Cauchy Problem for (\ref{1-1}).

For each $n\in\mathbf{N}$ and each $u\in C(M,\R)$, let

\[
\widetilde{T}_nu(x)=\inf_{\substack{k\in\mathbf{N} \\ n\leq k\leq
2n}}\inf_{\gamma}\Big\{u(\gamma(0))+\int_{0}^k
L(\gamma,\dot{\gamma},\sigma)d\sigma\Big\}
\]
for all $x\in M$, where the second infimum is taken among the continuous and piecewise $C^1$ paths $\gamma:[0,k]\rightarrow M$
with $\gamma(k)=x$. One can easily check that for each $n\in\mathbf{N}$, $\widetilde{T}_n$ is an operator from
$C(M,\R)$ to itself, and that $\{\widetilde{T}_n\}_{n\in\mathbf{N}}$ is a semigroup of operators.
We have, by definition, for each $\tau\in[0,1]$, each $n\in\mathbf{N}$, each $u\in C(M,\R)$ and each $x\in M$,

\[
\widetilde{T}^{\tau}_nu(x)=T_{\tau}\circ\widetilde{T}_nu(x)=\inf_{\substack{k\in\mathbf{N} \\ n\leq k\leq
2n}}T_{\tau+k}u(x).
\]

Given $n\in\mathbf{N}$ and $u\in C(M,\R)$, let 
\begin{equation}\label{V_def}
V^u_n(x,t):=T_t\circ\widetilde{T}_nu(x)
\end{equation}
for all $(x,t)\in M\times[0,+\infty)$. In view of the fact just mentioned in subsection 2.1, $V^u_n(x,t)$ is the unique viscosity solution of the equation (\ref{1-1}) with $V^u_n(x,0)=\widetilde{T}_nu(x)$, and thus satisfies (\ref{1-1}) at any point of differentiability. Obviously, $U^u_n=V^u_n{_{|_{M\times[0,1]}}}$. According to the convergence result of $\widetilde{T}^{\tau}_n$, we have

\begin{equation}\label{2-1}
\lim_{n\to+\infty}V_n^u(x,t)=\bar{u}(x,[t])
\end{equation}
uniformly on $(x,t)\in M\times[0,T]$ for all $T>0$.

For each $n\in\mathbf{N}$ and each $u\in C(M,\R)$, by definition, it is easy to see that

\[
\widetilde{T}_n u(x)= \inf_{\substack{k\in\mathbf{N} \\ n\leq k\leq
2n}}\inf_{\gamma}\Big\{u(\gamma(-k))+\int_{-k}^0
L(\gamma,\dot{\gamma},\sigma)d\sigma\Big\}
\]
for all $x\in M$, where the second infimum is taken among the continuous and piecewise $C^1$ paths $\gamma:[-k,0]\rightarrow
M$ with $\gamma(0)=x$. Therefore, for each $\tau\in[0,1]$, each $n\in\mathbf{N}$ and each $u\in C(M,\R)$, we have

\begin{equation}\label{2-2}
\widetilde{T}^\tau_n u(x)= \inf_{k\in\mathbf{N} \atop n\leq k\leq
2n}\inf_{\gamma}\Big\{u(\gamma(-k))+\int_{-k}^\tau
L(\gamma,\dot{\gamma},\sigma)d\sigma\Big\}
\end{equation}
for all $x\in M$, where the second infimum is taken among the continuous and piecewise $C^1$ paths $\gamma:[-k,\tau]\rightarrow
M$ with $\gamma(\tau)=x$. Thus, (\ref{2-2}) can be used as an equivalent definition of the new Lax-Oleinik operator associated with $L$.

Finally, given $a\in\R$, for each $\tau\in[0,1]$, each $n\in\mathbf{N}$ with $n\geq|\{a\}|$, let us then define the operator $\widetilde{T}^{\tau,a}_n:C(M,\R)\to C(M,\R)$ by
\begin{equation*}
\widetilde{T}^{\tau,a}_n u(x)= \inf_{k\in\mathbf{N} \atop n+\{a\}\leq k\leq
2n+\{a\}}\inf_{\gamma}\Big\{u(\gamma(-k))+\int_{-k}^\tau
L(\gamma,\dot{\gamma},\sigma)d\sigma\Big\}
\end{equation*}
for all $x\in M$, where the second infimum is taken among the continuous and piecewise $C^1$ paths $\gamma:[-k,\tau]\rightarrow
M$ with $\gamma(\tau)=x$. Let $U_n^{u,a}(x,\tau)=\widetilde{T}^{\tau,a}_n u(x)$. Then we have
\begin{equation}\label{2-3}
\lim_{n\to+\infty}U_n^{u,a}(x,\tau)=\bar{u}(x,[\tau])
\end{equation}
uniformly on $(x,\tau)\in M\times[0,1]$. The proof of (\ref{2-3}) is similar to that of Theorem 1.2 in \cite{WY1} and so it is omitted.

\subsubsection{Weak KAM solutions}
A function $w:M\times\mathbf{S}^1\to\R$ is called a subsolution of (\ref{1-1}) if it is Lipschitz and satisfies the inequality $d_tw+H(x,d_xw,t)\leq 0$ at almost every point. This definition is equivalent to the notion of viscosity subsolutions, see \cite{Fat-b}. A function $w:M\times\mathbf{S}^1\to\R$ is called a weak KAM solution of (\ref{1-1}) if $w$ is a subsolution of (\ref{1-1}) and if, for every $(x,[t])\in M\times\mathbf{S}^1$ there exists a curve $\gamma:(-\infty,[t]]\to M$ with $\gamma([t])=x$ such that

\begin{equation}\label{calibrated}
w(x,[t])-w(\gamma(t'),[t'])=\int_{t'}^{[t]}L(\gamma,\dot{\gamma},\sigma)d\sigma,\quad
\forall t'\in(-\infty,[t]].
\end{equation}
Such a curve is called a $(w,L,0)$-calibrated curve associated with $(x,[t])$.

Let $w$ be a weak KAM solution. Then it satisfies (\ref{1-1}) at any point of differentiability.
Given $(x,[t])\in M\times\mathbf{S}^1$, $w$ is differentiable at $(x,[t])$ if and only if there is a unique $(w,L,0)$-calibrated curve associated with $(x,[t])$. If $\gamma$ is a $(w,L,0)$-calibrated curve associated with $(x,[t])$, then $w$ is differentiable at $(\gamma(t'),[t'])$ and $d_xw(\gamma(t'),[t'])=\frac{\partial L}{\partial v}(\gamma(t'),\dot{\gamma}(t'),t')$ for all $t'<[t]$.

\subsubsection{Remarks on \cite{Bernard1} and \cite{Bernard2}}

A similar Lax-Oleinik operator has been defined and studied by P. Bernard in \cite{Bernard1, Bernard2} and, here, we recall such a construction underlying the common points with the new Lax-Oleinik operator in the study we are interested in.

Define the action functional associated with a $1$-periodic in time Tonelli Lagrangian function 
\[
\ell: TM \times \R \to \R
\]
as 
\begin{equation*}
F^{\ell}: [0, \infty) \times [0, \infty) \times M \times M \to \R
\end{equation*}
where
\begin{equation*}
F^{\ell}_{t, t'}(x, y) = \inf_{\gamma} \int_{t}^{t'} \ell(\gamma(s), \dot\gamma(s), s)\;ds
\end{equation*}
where the infimum is taken over the set of all continuous and piecewise $C^1$ curves $\gamma: [t, t'] \to M$ such that $\gamma(t) = x$ and $\gamma(t')=y$.

We stress that if the Aubry set associated with $\ell$ contains only one hyperbolic periodic orbit, then $\ell$ is regular, i.e., the infimum of the action of all closed curve is 0. So, we can study the asymptotic behavior of solutions to the corresponding Hamilton-Jacobi equation either with the new Lax-Oleinik semigroup either with the approach by P. Bernard. 

Indeed, if $\ell$ is regular then we have that the function 
\begin{equation*}
(t, t', x, x') \mapsto F^{\ell}(t, t', x, x')
\end{equation*}
is Lipschitz continuous and bounded on $\{(t, t') \in \R^2: t' \geq t +1\}$. So, given $(s, s') \in S^1 \times S^1$ we define the {\it action potential} as
\begin{equation*}
\Phi_{s, s'}(x, x') = \inf F_{t, t'}^{\ell}(x, x'), \quad \forall\; x, x' \in M \tag{\it action potential}
\end{equation*}
where the infimum is taken over all $t$, $t' \in \R^2$ such that $s = [ t ]$, $s' = [ t']$ and $t' \geq t+1$. Similarly, we define the {\it extended Peierls barrier} as
\begin{equation}\label{Ext_Peierls}
h_{s, s'}(x, x') = \liminf_{t-t' \to \infty} F_{t, t'}^{\ell}(x, x'),  \quad \forall\; x, x' \in M
\end{equation}
where the infimum is taken over all $t$, $t' \in \R^2$ such that $s = [ t]$, $s' = [ t' ]$.

In \cite{Bernard1} and \cite{Bernard2} the author proved that 
\begin{equation*}
\lim_{n \to \infty} F_{0, n+ \tau}(x, y) = h_{0, [ \tau ]} (x, y)
\end{equation*}
uniformly for $(\tau, x, y) \in [0,1] \times M \times M$ and, since 
\begin{equation*}
T_{n + \tau} \varphi(x) = \inf_{y \in M} \{\varphi(y) + F_{0, n+ \tau}(y, x)\} 
\end{equation*}
we have 
\begin{equation*}
\lim_{n \to \infty} T_{n+ \tau} \varphi(x) = \inf_{y \in M} \{\varphi(y) + h_{0, [ \tau ]}(y, x)\}. 
\end{equation*}

\section{Convergence of adherences in time-periodic case}\label{sec_thm1}

As a preliminary result, in order to prove \Cref{thm_2_intro} we need to show the adherences in the time periodic case. Thus, in this section, we fix a 1 time-periodic Lagrangian function $L: TM \times \mathbf{S}^1 \to \R$. 

We begin by introducing the following sets: for each $n\in\mathbf{N}$ and each $u\in C(M,\R)$, let

\[
G_n:=\Big\{\big(x,[\tau],d_xU_n^u(x,[\tau]),d_{\tau}U_n^u(x,[\tau])\big): (x,[\tau])\in \mbox{Dom} (dU_n^u) \Big\}
\]
with $G_n \subset T^*(M\times \mathbf{S}^1)$, and

\begin{equation}\label{limit_ade}
G:=\Big\{\big(x,[\tau],d_x\bar{u}(x,[\tau]),d_{\tau}\bar{u}(x,[\tau])\big): (x,[\tau])\in \mbox{Dom} (d\bar{u})\Big\}
\end{equation}
with $G \subset T^*(M\times \mathbf{S}^1)$. Then $\overline{G}_n$ and $\overline{G}$ are called the adherences of $G_n$ and $G$, respectively.

\begin{theorem}\label{thm_1}
For each $u\in C(M,\R)$, we have
\[
\lim_{n\to+\infty}d_H(\overline{G}_n,\overline{G})=0
\]
where $d_H$ denotes the Hausdorff metric\footnote{Let $(X,d)$ be a metric space and $\mathcal{K}(X)$ be the set of nonempty compact subsets of $X$. The Hausdorff metric $d_H$ is defined by

\[
d_H(K_1,K_2)=\max\left\{\sup_{x\in K_1}d(x,K_2),\sup_{x\in K_2}d(x,K_1)\right\}, \quad \forall K_1,\ K_2\in\mathcal{K}(X).
\]}.
\end{theorem}

\subsection{Preliminary lemmas}
The following two lemmas are useful in the proof of  \Cref{thm_1}.

\begin{lemma}\label{le1}
Given $u\in C(M,\R)$, let $\bar{u}=\displaystyle{\lim_{n\to+\infty}}U^u_n$. Then $\|d_x\bar{u}\|$ and $\|d_\tau\bar{u}\|$ are bounded. Moreover, $\|d_xU^u_n\|$ and $\|d_{\tau}U^u_n\|$ are bounded by a constant independent of $n\in\mathbf{N}\backslash\{0\}$.
\end{lemma}

\begin{proof}
Since $\bar{u}$ is a weak KAM solution, then it is Lipschitz and thus $\|d_x\bar{u}\|$ is bounded by the Lipschitz constant of $\bar{u}$. If $(x,[\tau])$ is a differentiability point of $\bar{u}$, then $d_\tau\bar{u}=-H(x,d_x\bar{u},[\tau])$. In view of the boundedness of $\|d_x\bar{u}\|$, $\|d_\tau\bar{u}\|$ is also bounded.

Note that

\begin{multline*}
|U^u_n(x,[\tau])-U^u_n(y,[\tau])|=\left|\inf_{k\in\mathbf{N} \atop n\leq k\leq 2n}T_{[\tau]+k}u(x)-\inf_{k\in\mathbf{N} \atop n\leq k\leq 2n}T_{[\tau]+k}u(y)\right|\\ \leq\sup_{k\in\mathbf{N} \atop n\leq k\leq 2n}\left|T_{[\tau]+k}u(x)-T_{[\tau]+k}u(y)\right|
\end{multline*}
for all $n\in\mathbf{N}$, $x$, $y\in M$, $\tau\in[0,1]$. From a result of Fathi \cite{Fat-b}, there exists a constant $K(1)>0$ such that $T_{[\tau]+k}u$ is Lipschitz with Lipschitz constant $\leq K(1)$, where $K(1)$ is independent of $u$, $n\in\mathbf{N}\backslash\{0\}$ and $\tau\in[0,1]$. Therefore, we have

\[
|U^u_n(x,[\tau])-U^u_n(y,[\tau])|\leq\sup_{k\in\mathbf{N} \atop n\leq k\leq 2n}|T_{[\tau]+k}u(x)-T_{[\tau]+k}u(y)|\leq K(1)d(x,y)
\]
for all $n\in\mathbf{N}\backslash\{0\}$, $x$, $y\in M$ and $\tau\in[0,1]$, which implies the boundedness of $\|d_xU^u_n\|$. Since $U^u_n$ satisfies the equation (\ref{1-1}) at any point of differentiability, then the boundedness of $\|d_xU^u_n\|$ implies the boundedness of $\|d_{\tau}U^u_n\|$.
\end{proof}

\begin{lemma}\label{le2}
Given $(x,[\tau])\in M\times\mathbf{S}^1$, let $w$ be a  weak KAM solution and let $\gamma:(-\infty,[\tau]]\to M$ with $\gamma([\tau])=x$ be a $(w,L,0)$-calibrated curve associated with $(x,[\tau])$. Set $v=\dot{\gamma}([\tau])$, $p=\frac{\partial L}{\partial v}(x,v,[\tau])$, $e=-H(x,p,[\tau])$. Then $(x,[\tau],p,e)\in \overline{G}$.
\end{lemma}

\begin{proof}
For each $t<[\tau]$, since $w$ is differentiable at $(\gamma(t),[t])$ and $d_xw(\gamma(t),[t])=\frac{\partial L}{\partial v}(\gamma(t),\dot{\gamma}(t),t)$, then $d_tw(\gamma(t),[t])=-H(\gamma(t),\frac{\partial L}{\partial v}(\gamma(t),\dot{\gamma}(t),t),t)$ and

\begin{multline*}
(\gamma(t),[t],\frac{\partial L}{\partial v}(\gamma(t),\dot{\gamma}(t),t),-H(\gamma(t),\frac{\partial L}{\partial v}(\gamma(t),\dot{\gamma}(t),t),t))\\ =(\gamma(t),[t],d_xw(\gamma(t),[t]),d_tw(\gamma(t),[t]))\in G.
\end{multline*}
If we let $t\to [\tau]$, we see that

\[
(\gamma(t),[t],d_xw(\gamma(t),[t]),d_tw(\gamma(t),[t]))\to (x,[\tau],p,e),
\]
which implies that $(x,[\tau],p,e)\in \overline{G}$.
\end{proof}

\subsection{Proof of \Cref{thm_1}}

%\noindent\emph{Proof of \Cref{thm_1}}. 
Our purpose is to show that for each $\varepsilon>0$, there exists $N_0\in\mathbf{N}$ such that:
\begin{itemize}
 \item[$(i)$] $\overline{G}_n\subset \overline{G}^\varepsilon$; 
 \item[$(ii)$] $\overline{G}\subset \overline{G}_n^\varepsilon$ for all $n\geq N_0$, $n\in\mathbf{N}$.
 \end{itemize}

\vskip0.1cm

\fbox{{\bf Step 1.}} We first prove (i) by contradiction. Otherwise, there would be $\delta_1>0$ and a sequence $\{(x_n,[\tau_n])\}_n\subset M\times \mathbf{S}^1$ of differentiability points of $U^u_n$, such that

\begin{equation}\label{3-1}
(x_n,[\tau_n],d_xU^u_n(x_n,[\tau_n]),d_\tau U^u_n(x_n,[\tau_n]))\not\in \overline{G}^{\delta_1}, \quad \forall n\in\mathbf{N}.
\end{equation}
Let $p_n=d_xU^u_n(x_n,[\tau_n])$, $e_n=d_\tau U^u_n(x_n,[\tau_n])=-H(x_n,p_n,[\tau_n])$. From Lemma \ref{le1}
we conclude that $\{(x_n,[\tau_n],p_n,e_n)\}_n$ are contained in a compact subset of $T^*(M\times\mathbf{S}^1)$.
So we may assume upon passing if necessary to a subsequence that $(x_n,[\tau_n],p_n,e_n)\to(x_0,[\tau_0],p_0,e_0)$ as $n\to+\infty$. Obviously, $e_0=-H(x_0,p_0,[\tau_0])$. We assert that $(x_0,[\tau_0],p_0,e_0)\in\overline{G}$, which contradicts (\ref{3-1}). This contradiction proves (i).

Our task is now to show that $(x_0,[\tau_0],p_0,e_0)\in\overline{G}$. Let $(\gamma(s),\dot{\gamma}(s),[s])=\phi^L_{s-[\tau_0]}(x_0,v_0,[\tau_0])$, $s\in(-\infty,[\tau_0]]$, where $p_0=\frac{\partial L}{\partial v}(x_0,v_0,[\tau_0])$.
We assert that $\gamma$ is a $(\bar{u},L,0)$-calibrated curve associated with $(x_0,[\tau_0])$.
If this assertion is true, then by Lemma \ref{le2}, we deduce that $(x_0,[\tau_0],p_0,e_0)\in\overline{G}$.
Hence (i) will be proved by showing that $\gamma$ is a $(\bar{u},L,0)$-calibrated curve associated with $(x_0,[\tau_0])$, i.e.,

\begin{equation}\label{3-2}
\bar{u}(x_0,[\tau_0])-\bar{u}(\gamma(a),[a])=\int_a^{[\tau_0]}L(\gamma,\dot{\gamma},s)ds
\end{equation}
for all $a<[\tau_0]$.

For each $(x_n,[\tau_n])$, by the definition of $U^u_n$, there exist $k_n\in\mathbf{N}$ with $n\leq k_n\leq 2n$, and a minimizing extremal curve $\gamma_n:[-k_n,[\tau_n]]\to M$ with $\gamma_n([\tau_n])=x_n$ such that

\begin{equation}\label{3-3}
U^u_n(x_n,[\tau_n])=u(\gamma_n(-k_n))+\int_{-k_n}^{[\tau_n]}L(\gamma_n,\dot{\gamma}_n,s)ds.
\end{equation}
Since $(x_n,[\tau_n])$ is a differentiability point of $U_n^u$ and $\gamma_n$ satisfies (\ref{3-3}), then we have
$p_n=d_xU^u_n(x_n,[\tau_n])=\frac{\partial L}{\partial v}(\gamma_n([\tau_n]),\dot{\gamma}_n([\tau_n]),[\tau_n])$.
And thus $(\gamma_n(s),\dot{\gamma}_n(s),[s])=\phi^L_{s-[\tau_n]}(x_n,v_n,[\tau_n])$, $s\in[-k_n,[\tau_n]]$, where $p_n=\frac{\partial L}{\partial v}(x_n,v_n,[\tau_n])$.

An outline of the proof of (\ref{3-2}) is as follows. First, we show that given $a<[\tau_0]$,

\begin{equation}\label{3-4}
U^u_n(x_n,[\tau_n])-U^{u,a}_n(\gamma_n(a),[a])=\int_a^{[\tau_n]}L(\gamma_n,\dot{\gamma}_n,s)ds
\end{equation}
for $n\in\mathbf{N}$ large enough. Second, we prove the following equalities

\begin{align}
\lim_{n\to+\infty}U^u_n(x_n,[\tau_n])=&\;\bar{u}(x_0,[\tau_0]),\\
\lim_{n\to+\infty}U^{u,a}_n(\gamma_n(a),[a])=&\;\bar{u}(\gamma(a),[a]),\\
\lim_{n\to+\infty}\int_a^{[\tau_n]}L(\gamma_n,\dot{\gamma}_n,s)ds=&\;\int_a^{[\tau_0]}L(\gamma,\dot{\gamma},s)ds.
\end{align}
Finally, combining (\ref{3-4})-(3.7) gives the desired equality (\ref{3-2}).

We are now in a position to prove (\ref{3-4}). Given $a<[\tau_0]$, for $n\in\mathbf{N}$ large enough, from the definition of $U^{u,a}_n$ we have

\[
U^{u,a}_n(\gamma_n(a),[a])=\inf_{k\in\mathbf{N} \atop n+\{a\}\leq k\leq
2n+\{a\}}\inf_{\alpha}\Big\{u(\alpha(-k))+\int_{-k}^{[a]}
L(\alpha,\dot{\alpha},\sigma)d\sigma\Big\},
\]
where the second infimum is taken among the continuous and piecewise $C^1$ paths $\alpha:[-k,[a]]\rightarrow
M$ with $\alpha([a])=\gamma_n(a)$. Define a curve $\alpha_n:[-k_n-\{a\},[a]]\to M$ by $\alpha_n(\sigma)=\gamma_n(\sigma+\{a\})$. Then $\alpha_n([a])=\gamma_n(a)$ and

\begin{equation}\label{3-8}
u(\gamma_n(-k_n))+\int_{-k_n}^aL(\gamma_n,\dot{\gamma}_n,s)ds=u(\alpha_n(-k_n-\{a\}))+\int_{-k_n-\{a\}}^{[a]}
L(\alpha_n,\dot{\alpha}_n,\sigma)d\sigma.
\end{equation}
We assert that

\begin{equation}\label{3-9}
U^{u,a}_n(\gamma_n(a),[a])=u(\alpha_n(-k_n-\{a\}))+\int_{-k_n-\{a\}}^{[a]}
L(\alpha_n,\dot{\alpha}_n,\sigma)d\sigma.
\end{equation}
To prove (\ref{3-9}), we argue by contradiction. For, otherwise, there would be $h_n\in\mathbf{N}$ with
$n+\{a\}\leq h_n\leq 2n+\{a\}$, and a curve $\beta_n:[-h_n,[a]]\to M$ with $\beta_n([a])=\gamma_n(a)$ such that

\begin{equation}\label{3-10}
u(\beta_n(-h_n))+\int_{-h_n}^{[a]}L(\beta_n,\dot{\beta}_n,\sigma)d\sigma<u(\alpha_n(-k_n-\{a\}))+\int_{-k_n-\{a\}}^{[a]}
L(\alpha_n,\dot{\alpha}_n,\sigma)d\sigma.
\end{equation}
Define a curve $\bar{\gamma}_n:[-h_n+\{a\},a]\to M$ by $\bar{\gamma}_n(s)=\beta_n(s-\{a\})$. Then, from (\ref{3-10}) and (\ref{3-8}) we have

\begin{equation}\label{3-11}
u(\bar{\gamma}_n(-h_n+\{a\}))+\int_{-h_n+\{a\}}^aL(\bar{\gamma}_n,\dot{\bar{\gamma}}_n,s)ds<
u(\gamma_n(-k_n))+\int_{-k_n}^aL(\gamma_n,\dot{\gamma}_n,s)ds.
\end{equation}
Since $n+\{a\}\leq h_n\leq 2n+\{a\}$, then $n\leq h_n-\{a\}\leq 2n$. Consider the curve $\tilde{\gamma}_n:[-h_n+\{a\},[\tau_n]]\to M$ defined by

\[
\tilde{\gamma}_n(s)=\left\{\begin{array}{ll}
          \bar{\gamma}_n(s),\ & s\in[-h_n+\{a\},a],\\[2mm]
          \gamma_n(s),\ & s\in[a,[\tau_n]].
\end{array}\right.
\]
In view of (\ref{3-11}), we have

\begin{align*}
&u(\tilde{\gamma}_n(-h_n+\{a\}))+\int_{-h_n+\{a\}}^{[\tau_n]}L(\tilde{\gamma}_n,\dot{\tilde{\gamma}}_n,s)ds\\
=&\;u(\bar{\gamma}_n(-h_n+\{a\}))+\int_{-h_n+\{a\}}^aL(\bar{\gamma}_n,\dot{\bar{\gamma}}_n,s)ds+\int_a^{[\tau_n]}
L(\gamma_n,\dot{\gamma}_n,s)ds\\
<&\;u(\gamma_n(-k_n))+\int_{-k_n}^{[\tau_n]}
L(\gamma_n,\dot{\gamma}_n,s)ds,
\end{align*}
which contradicts the minimality of $\gamma_n$. This contradiction shows that (\ref{3-9}) holds.
The desired equality (\ref{3-4}) follows from (\ref{3-3}), (\ref{3-8}) and (\ref{3-9}).

Next we want to prove the equalities (3.5)-(3.7). (3.5) follows immediately from (\ref{2-1}), the Lipschitz property of $\bar{u}$ and the following inequality
\[ |U^u_n(x_n,[\tau_n])-\bar{u}(x_0,[\tau_0])|\leq|U^u_n(x_n,[\tau_n])-\bar{u}(x_n,[\tau_n])|+|\bar{u}(x_n,[\tau_n])-\bar{u}(x_0,[\tau_0])|.
\]

To prove (3.6), note that

\begin{multline}\label{3-12}
|U^{u,a}_n(\gamma_n(a),[a])-\bar{u}(\gamma(a),[a])|\leq|U^{u,a}_n(\gamma_n(a),[a])-\bar{u}(\gamma_n(a),[a])|
\\ +|\bar{u}(\gamma_n(a),[a])-\bar{u}(\gamma(a),[a])|.
\end{multline}
By (\ref{2-3}), we have $\lim_{n\to+\infty}U^{u,a}_n(\gamma_n(a),[a])=\bar{u}(\gamma(a),[a])$.
If

\begin{equation}\label{3-13}
\lim_{n\to+\infty}\bar{u}(\gamma_n(a),[a])=\bar{u}(\gamma(a),[a]),
\end{equation}
then from (\ref{3-12}), we conclude that (3.6) holds. To prove (\ref{3-13}), it is sufficient to show that

\begin{equation}\label{3-14}
d(\gamma_n(a),\gamma(a))\to 0, \quad n\to+\infty.
\end{equation}
Since $(x_n,v_n,[\tau_n])\to (x_0,v_0,[\tau_0])$ as $n\to+\infty$, then by the continuity of the solutions of the
Euler-Lagrange equation with respect to initial values, we have

\begin{equation}\label{3-15}
d(\gamma_n([\tau_n]-b),\gamma([\tau_0]-b))\to 0, \quad n\to+\infty,
\end{equation}
where $b=[\tau_0]-a$. In view of the a priori compactness given by Lemma 3.4 in \cite{WY1},
we have $(\gamma_n(s),\dot{\gamma}_n(s),[s])\in \mathcal{C}_1$, $\forall s\in[-k_n,[\tau_n]]$, $\forall n\in\mathbf{N}\backslash\{0\}$, where $\mathcal{C}_1$ is a compact subset of $TM\times\mathbf{S}^1$. Consequently, we obtain $d(\gamma_n(a),\gamma_n([\tau_n]-b))\leq A |a-[\tau_n]+b|$ for some constant $A>0$,
which implies that

\begin{equation}\label{3-16}
d(\gamma_n(a),\gamma_n([\tau_n]-b))\to 0, \quad n\to+\infty.
\end{equation}
Note that

\[
d(\gamma_n(a),\gamma(a))\leq d(\gamma_n(a),\gamma_n([\tau_n]-b))+d(\gamma_n([\tau_n]-b),\gamma([\tau_0]-b)),
\]
which together with (\ref{3-15}) and (\ref{3-16}) yields (\ref{3-14}).

In order to prove (3.7), note that

\begin{multline*}
\left|\int_a^{[\tau_n]}L(\gamma_n,\dot{\gamma}_n,s)ds-\int_a^{[\tau_0]}L(\gamma,\dot{\gamma},s)ds\right|\\
\leq \left|\int_a^{[\tau_n]}L(\gamma_n,\dot{\gamma}_n,s)ds-\int_{[\tau_n]-b}^{[\tau_n]}L(\gamma_n,\dot{\gamma}_n,s)ds\right|
+\left|\int_{[\tau_n]-b}^{[\tau_n]}L(\gamma_n,\dot{\gamma}_n,s)ds-\int_a^{[\tau_0]}L(\gamma,\dot{\gamma},s)ds\right|\\
=\left|\int_a^{[\tau_0]}L(\gamma_n(\sigma+[\tau_n]-[\tau_0]),\dot{\gamma}_n(\sigma+[\tau_n]-[\tau_0]),\sigma +[\tau_n]-[\tau_0])d\sigma-\int_a^{[\tau_0]}L(\gamma,\dot{\gamma},s)ds\right|
\\ +\left|\int_{[\tau_n]-b}^aL(\gamma_n,\dot{\gamma}_n,s)ds\right|.
\end{multline*}
Since $(x_n,v_n,[\tau_n])\to (x_0,v_0,[\tau_0])$ as $n\to+\infty$ and $(\gamma_n(s),\dot{\gamma}_n(s),[s])\in \mathcal{C}_1$, $\forall s\in[-k_n,[\tau_n]]$, $\forall n\in\mathbf{N}\backslash\{0\}$, then by the continuity of the solutions of the Euler-Lagrange equation with respect to initial values, we conclude that (3.7) holds.

\vskip0.2cm

\fbox{{\bf Step 2.}} Now we prove (ii) by contradiction. Otherwise, there would be $\delta_2>0$ and a sequence $\{(x_n,[\tau_n])\}_n\subset M\times\mathbf{S}^1$ of differentiability points of $\bar{u}$ such that

\begin{equation}\label{3-17}
\overline{G}_n\cap\{(x_n,[\tau_n],d_x\bar{u}(x_n,[\tau_n]),d_\tau\bar{u}(x_n,[\tau_n]))\}^{\delta_2}=\emptyset,\ \forall n\in\mathbf{N}.
\end{equation}
Sending $n\to+\infty$, by Lemma \ref{le1} we may assume, passing if necessary to subsequence, that

\begin{equation}\label{3-18}
(x_n,[\tau_n],d_x\bar{u}(x_n,[\tau_n]),d_\tau\bar{u}(x_n,[\tau_n]))\to(\bar{x},[\bar{\tau}],\bar{p},\bar{e})\in T^*(M\times\mathbf{S}^1).
\end{equation}
Since $\bar{u}$ and $U^u_n$ are both locally Lipschitz, then from Lemma \ref{le1}, we have

\[
\pi^*(\overline{G})=M\times\mathbf{S}^1,\quad \pi^*(\overline{G}_n)=M\times\mathbf{S}^1, \quad\forall n\in\mathbf{N}\backslash\{0\},
\]
where $\pi^*:T^*(M\times\mathbf{S}^1)\to M\times\mathbf{S}^1$ denotes the projection. From (i) there exists $N_1\in\mathbf{N}$ such that $\overline{G}_n\subset\overline{G}^{\frac{\delta_2}{2}}$ for all
$n\geq N_1$, $n\in\mathbf{N}$. Therefore, we have

\[
\emptyset\neq\overline{G}_n|_{(\bar{x},[\bar{\tau}])}\subset(\overline{G}|_{(\bar{x},[\bar{\tau}])})^{\frac{\delta_2}{2}},\quad
\forall n\geq N_1, \ n\in\mathbf{N},
\]
where $\overline{G}_n|_{(\bar{x},[\bar{\tau}])}=(\pi^*|_{\overline{G}_n})^{-1}(\bar{x},[\bar{\tau}])$ and
$\overline{G}|_{(\bar{x},[\bar{\tau}])}=(\pi^*|_{\overline{G}})^{-1}(\bar{x},[\bar{\tau}])$.

Suppose that $(\bar{x},[\bar{\tau}])$ is a differentiability point of $\bar{u}$. Then

\[
\overline{G}|_{(\bar{x},[\bar{\tau}])}=\{(\bar{x},[\bar{\tau}],d_x\bar{u}(\bar{x},[\bar{\tau}]),d_\tau\bar{u}(\bar{x},[\bar{\tau}]))\}.
\]
It is not hard to see that $\bar{p}=d_x\bar{u}(\bar{x},[\bar{\tau}])$ and $\bar{e}=d_\tau\bar{u}(\bar{x},[\bar{\tau}])$. By (\ref{3-18}) there exists $N_2\in\mathbf{N}$ such that

\[
d((x_n,[\tau_n],d_x\bar{u}(x_n,[\tau_n]),d_\tau\bar{u}(x_n,[\tau_n])),(\bar{x},[\bar{\tau}],\bar{p},\bar{e}))<\frac{\delta_2}{2}
\]
for all $n\geq N_2$, $n\in\mathbf{N}$. Therefore,

\[
\emptyset\neq\overline{G}_n|_{(\bar{x},[\bar{\tau}])}\subset(\overline{G}|_{(\bar{x},[\bar{\tau}])})^{\frac{\delta_2}{2}}
=\{(\bar{x},[\bar{\tau}],\bar{p},\bar{e})\}^{\frac{\delta_2}{2}}\subset\{(x_n,[\tau_n],d_x\bar{u}(x_n,[\tau_n]),d_\tau\bar{u}(x_n,[\tau_n]))\}^{\delta_2}
\]
for all $n\geq\max\{N_1,N_2\}$, $n\in\mathbf{N}$, which contradicts (\ref{3-17}).

Suppose that $(\bar{x},[\bar{\tau}])$ is not a differentiability point of $\bar{u}$. In view of (\ref{3-18}),
we have $(\bar{x},[\bar{\tau}],\bar{p},\bar{e})\in\overline{G}$. Hence, there exists a $(\bar{u},L,0)$-calibrated curve $\gamma:(-\infty,[\bar{\tau}]]\to M$ associated with $(\bar{x},[\bar{\tau}])$ such that

\[
\bar{p}=\frac{\partial L}{\partial v}(\gamma([\bar{\tau}]),\dot{\gamma}([\bar{\tau}]),[\bar{\tau}]),\quad \bar{e}=-H(\bar{x},\bar{p},[\bar{\tau}]).
\]
Take $t'<[\bar{\tau}]$ close enough to $[\bar{\tau}]$ so that

\begin{equation}\label{3-19}
d((\gamma(t'),t',\frac{\partial L}{\partial v}(\gamma(t'),\dot{\gamma}(t'),t'),-H(\gamma(t'),\frac{\partial L}{\partial v}(\gamma(t'),\dot{\gamma}(t'),t'),t')),(\bar{x},[\bar{\tau}],\bar{p},\bar{e}))<\frac{\delta_2}{4}.
\end{equation}
Set $x'=\gamma(t')$, $p'=\frac{\partial L}{\partial v}(\gamma(t'),\dot{\gamma}(t'),t')$, $e'=-H(\gamma(t'),\frac{\partial L}{\partial v}(\gamma(t'),\dot{\gamma}(t'),t'),t')$. We can then rewrite (\ref{3-19}) as

\[
d((x',t',p',e'),(\bar{x},[\bar{\tau}],\bar{p},\bar{e}))<\frac{\delta_2}{4}.
\]
From (i) there exists $N_3\in\mathbf{N}$ such that $\overline{G}_n\subset\overline{G}^{\frac{\delta_2}{2}}$ for all $n\geq N_3$, $n\in\mathbf{N}$. Therefore,

\[
\emptyset\neq\overline{G}_n|_{(x',t')}\subset(\overline{G}|_{(x',t')})^{\frac{\delta_2}{2}}
=\{(x',t',d_x\bar{u}(x',t'),d_\tau\bar{u}(x',t'))\}^{\frac{\delta_2}{2}}=\{(x',t',p',e')\}^{\frac{\delta_2}{2}}.
\]
From (\ref{3-18}) there exists $N_4\in\mathbf{N}$ such that

\[
d((x_n,[\tau_n],d_x\bar{u}(x_n,[\tau_n]),d_\tau\bar{u}(x_n,[\tau_n])),(\bar{x},[\bar{\tau}],\bar{p},\bar{e}))<\frac{\delta_2}{8}
\]
for all $n\geq N_4$, $n\in\mathbf{N}$. Therefore,

\[
\emptyset\neq\overline{G}_n|_{(x',t')}\subset\{(x',t',p',e')\}^{\frac{\delta_2}{2}}\subset\{(x_n,[\tau_n],d_x\bar{u}(x_n,[\tau_n]),d_\tau\bar{u}(x_n,[\tau_n]))\}^{\delta_2}
\]
for all $n\geq\max\{N_3,N_4\}$, $n\in\mathbf{N}$, which is contrary to (\ref{3-17}). The proof of \Cref{thm_1} is thus complete. \qed

\section{Proof of \Cref{thm_2_intro}}\label{sec_thm2}

Hereafter we consider $L_1: \R \times TM \to \R$ a time dependent Lagrangian function such that 
\begin{equation*}
\| L_1(t + n, x, v) - \overline{L_1}(t, x, v) \|_{C^2_c( T M\times \R;\R)} \leq C e^{-\rho n}, \mbox{ for some } C\in\mathbb{R},\;\rho > 0\; \mbox{ and }\; \forall\; n \in \mathbf{N}.
\end{equation*}

\subsection{Convergence to periodic solutions}

The next results can be proved by an easy adaptation of the one in \cite[Proposition 2.1, Proposition 2.5]{Sci}.

\begin{lemma}
For fixed $t'>0$ there exists $t_0>0$ such that 
\begin{equation*}
-e^{-\rho t'} t' + \inf_{x \in M} \left(F^{L_1}_{0, t}(y, x) + F^{\overline{L}}_{0, t'}(x, z)  \right) \leq F^{L_1}_{0, t+t'}(y, z) \leq e^{-\rho t'} t' + \inf_{x \in M} \left(F^{L_1}_{0, t}(y, x) + F^{\overline{L}}_{0, t'}(x, z)  \right)
\end{equation*}
for any $t \in (t_0, \infty) $ and $y$, $z \in M$, where $\rho$ is given in \eqref{rate_Lagrangian}. 
\end{lemma}

\begin{lemma}
For any given $t' > 0$ there exists $t_0 > 0$ such that 
\begin{equation*}
|T^{1}_{t+t'} \varphi - \overline{T}_{t'} \circ T^{1}_t \varphi| \leq e^{-\rho t'} t'
\end{equation*}
for any $\varphi \in C(M; \R)$ and any $t > t_0$, where $T^{1}_t$ and $\overline{T}_t$ denote the Lax-Oleinik semigroup associated with $L_1$ and $\overline{L}$, respectively. 
\end{lemma}

Next, let us define a new evolutive operator $\mathcal{T}_t: C(M; \R) \to \R$ by 
\begin{equation}\label{new_Lax}
\mathcal{T}_t \varphi(x) = T^{1}_t \varphi(x) - \inf_{x \in M} T^{1}_t \varphi(x), \quad (t > 0)
\end{equation}
where we recall that $T^{1}_t$ denotes the Lax-Oleinik semigroup associated with the time-dependent non-periodic Lagrangian function $L_1$. 
%So, we immediately have that 
%\begin{equation*}
%\mathcal{T}_t \varphi(x) \varphi(x) = z(x, t) - \inf_{x \in M} z(x, t)
%\end{equation*} 
%where $z: M \times [0, \infty) \to \R$ solves
%\begin{equation*}
%\begin{cases}
%\partial_t z(x, t) + \widehat{H}(x, d_x z(x, t), t)= 0, & (x, t) \in M \times [0, \infty)
%\\
%z(x, 0) = \varphi(x), & x \in M
%\end{cases}
%\end{equation*}
%in the viscosity sense. 

Next, by adapting the reasoning in \cite[Proposition 2.6]{Sci} we get the following.

\begin{proposition}\label{prop}
Let $\varphi \in C(M; \R)$. Then, the following hold.
\begin{enumerate}
\item For any $t > 0$ and any $x \in M$, $\mathcal{T}_t \varphi(x)$ is finite. 
\item For any $t > 0$ we have 
\[
\mathcal{T}_t \varphi(x) \geq 0
\]
and
\begin{equation*}
\inf_{y \in M} \inf_{z \in M} \left(F^{L_1}_{0, t}(y, x) - F^{L_1}_{0, t}(y, z) \right) \leq \mathcal{T}_t \varphi(x) \leq \sup_{y \in M} \inf_{z \in M} \left(F^{L_1}_{0, t}(y, x) - F^{L_1}_{0, t}(y, z) \right).
\end{equation*}
\item For any $t' > 0$ and any $\eps > 0$ there exists$t_0 > 0$ such that 
\begin{equation*}
\mathcal{T}_{t+t'} \varphi(x) \geq \inf_{y \in M} \inf_{z \in M} \left(F^{\overline{L}}_{0, t}(y, x) - F^{\overline{L}}_{0, t}(y, z) \right) - 2 e^{-\rho t'} t'
\end{equation*}
and
\begin{equation*}
\mathcal{T}_{t+t'} \varphi(x) \leq \sup_{y \in M} \inf_{z \in M} \left(F^{\overline{L}}_{0, t}(y, x) - F^{\overline{L}}_{0, t}(y, z) \right) + 2 e^{-\rho t'} t'
\end{equation*}
for any $t > t_0$. 
\item For any $t' > 0$ there exists $\eps(t') > 0$, with $\displaystyle{\lim_{t' \to 0}} \eps(t') = 0$, and $t_0 > 0$ such that for any $t > 0$ we have 
\begin{equation}\label{limit_est}
\left|\mathcal{T}_{t+t'} \varphi(x) - \overline{T}_{t'} \circ \mathcal{T}_{t} \varphi(x) + \inf_{x \in M}\overline{T}_{t'} \circ \mathcal{T}_{t} \varphi(x)  \right| \leq  e^{-\rho t'},
\end{equation}
where $\overline{T}_t$ denotes the Lax-Oleinik semigroup associated with the limiting periodic Lagrangian $\overline{L_1}$.
\item The function $(t, x) \mapsto \mathcal{T}_t \varphi(x)$ is continuous on $[0, \infty) \times M$ and for any $t_0 > 0$ it is equi-Lipschitz on $[t_0, \infty) \times M$. 
\end{enumerate} 
\end{proposition}

\subsubsection{Proof of ($i$) in \Cref{thm_2_intro}}

From (5.) in \Cref{prop}, appealing to Ascoli-Arzela theorem there exists $\{t_n\}_{n \in N}$  and $z_{\infty} \in C(M; \R)$ such that $t_n \uparrow \infty$ as $n \uparrow \infty$ and 
\begin{equation*}
\lim_{n \to \infty} \mathcal{T}_{t_n} z_0(x) = z_{\infty}(x). 
\end{equation*}
On the other hand, by \eqref{limit_est} in \Cref{prop} we have
\begin{multline}\label{nm_limit}
\lim_{m \to \infty \atop n \to \infty}\mathcal{T}_{\tau + m + t_n} z_0(x)= \lim_{m \to \infty \atop n \to \infty} \left(\overline{T}_{m+\tau} \circ \mathcal{T}_{t_n} z_0(x) - \inf_{x \in M} \overline{T}_{m+\tau} \circ \mathcal{T}_{t_n} z_0(x) \right)
\\
= \lim_{m \to \infty} \left(\overline{T}_{m+\tau} \circ \lim_{n \to \infty}\mathcal{T}_{t_n} z_0(x) - \inf_{x \in M} \overline{T}_{m+\tau} \circ \lim_{n \to \infty} \mathcal{T}_{t_n} z_0(x) \right)
\\
=  \lim_{m \to \infty} \left(\overline{T}_{m+\tau} z_{\infty}(x) - \inf_{x \in M}   \lim_{m \to \infty} \overline{T}_{m+\tau} z_{\infty}(x) \right)
= \overline{u}(x, \tau) - \inf_{x \in M} \overline{u}(x, \tau).
\end{multline}
Hence, setting 
\begin{equation*}
w(x, \tau) :=\lim_{m \to \infty \atop n \to \infty}\mathcal{T}_{\tau + t_n + m} z_0(x)
\end{equation*}
we get
\begin{equation}\label{u_inf}
w(x, \tau) = \overline{u}(x, \tau) - \inf_{x \in M} \overline{u}(x, \tau). 
\end{equation}
So, by construction of $\overline{u}$ in \eqref{u_bar} we have that the function $w(x, \cdot)$ is a $1$-periodic viscosity solution of \eqref{1-1}.
Next, we proceed to show the exponential rate of convergence. We first recall that from (4.) in \Cref{prop} we know that
\begin{equation*}
\left|\mathcal{T}_{\tau + t_n + m} z_0(x) - \left(\overline{T}_{n+\tau} \circ \mathcal{T}_{m} z_0(x) - \inf_{x \in M} \overline{T}_{n+\tau} \circ \mathcal{T}_{m} z_0(x) \right) \right| \leq e^{-\rho (m+t_n)}.
\end{equation*}
%Recalling that from \cite{W1} there exists $\rho > 0$, $K > 0$ such that
%\begin{equation*}
%\left\| \overline{T}_{n + \tau} z_0(x) - \overline{u}(x, \tau) \right\|_{\infty} \leq K e^{-\rho n}, \quad \forall\; n \in \mathbf{N}
%\end{equation*} 
%following the arguments in \eqref{nm_limit} we can set
%\begin{equation*}
%\eps(\tau, n) := e^{-\rho n}
%\end{equation*}
%independently of $\tau \in [0,1]$, up to a renormalization of the function. 
Hence, we get \eqref{exp_1} which completes the proof. \qed

\subsection{Convergence of adherences}

Given any Tonelli Lagrangian $L: TM \times \R \to \R$ let 
\[
B(k) = \{(x, v) \in TM: \|v\| \leq k \}
\]
and for any $t > 0 $ define
\begin{equation*}
U_t^L : C(M; \R) \times B(k) \to \R
\end{equation*}
by
\begin{equation*}
U_t^L (\varphi, (x, v)) = \varphi(\gamma_{x, v}(-t)) + \int_{-t}^{0} L(\gamma_{x, v}(s), \dot\gamma_{x, v}(s), s + t)\;ds.
\end{equation*} 
Then, it is easy to see that 
\begin{equation*}
T_t \varphi (x) = \min\{U_t (\varphi, (x, v)): (x, v) \in B(k)\}. 
\end{equation*}

Finally, we define 
\begin{equation*}
M_{L, x, t}(\varphi) = \left\{v \in B(k)_{|_{T_x M}} : U_t(\varphi, (x, v)) =\min_{w \in T_x M} U_t\left(\varphi, (x, w)\right) \right\}
\end{equation*}
and thus we have
\begin{equation*}
G\left(dT_t \varphi\right) = \mathcal{L}\left(\bigcup_{x \in M, \; \sharp \left(M_{L, x, t}(\varphi)\right) = 1} M_{L, x, t}(\varphi) \right).
\end{equation*}
where $\sharp M_{L, x, t}(\varphi)$ denotes the cardinality of the set $M_{L, x, t}(\varphi)$. Note that, the above characterization only consider the gradient w.r.t. the space variable of the value function.

From \cite[Lemma 3.1]{Sci} we have the following:
\begin{equation}\label{renorm}
\lim_{n \to \infty} \left(\min_{x \in M} T^1_{n-t}\varphi(x) -  \min_{x \in M} T^1_{\tau+n}\varphi(x)\right) = 0, \mbox{ for any } \tau \in [0,1] \mbox{ and } 0 < t < \langle \tau\rangle
\end{equation}
by the normalization assumption that the critical value is zero. 

\begin{proposition}\label{M_convergence}
For any $\eps > 0$, any $t > \delta$, for some $\delta > 0$, and any $x \in \mbox{Dom}(dw(\cdot, t)$, where $w$ is the limiting function defined in ($i$) in \Cref{thm_2_intro}, there exists $t_{\eps} > 0$ such that 
\begin{equation*}
\rho\left(M_{L_1, x, t}(\varphi), M_{\overline{L}, x, t}(\varphi) \right) \leq \eps. 
\end{equation*}
\end{proposition}
\proof 

Assume, by contradiction, that for any $t > t_0$ there exists $\eps_0 > 0$ such that 
\[
\sup_{v \in M_{L_1, x, t}(\varphi)} d(v, M_{\overline{L}, x, t}(\varphi)) \geq \eps{_0}. 
\]
Then, for any $t > \delta$, for $\delta>0$ suffinciently small, and any $x \in \mbox{Dom}(dw(\cdot, t))$, there exist $v_{x, t} \in M_{L_1, x, t}(\varphi)$ and $\overline{v}_{x, t} \in M_{\overline{L}, x, t}(\varphi)$ such that $d(v_{x, t}, \overline{v}_{x, t}) \geq \eps{_0}/2$. 
For any $n > n_0$, for $n_0 \in \mathbf{N}$ large enough, and any $\tau \in [0,1]$ there exists a minimal curve $\gamma_n: [-\tau-n, 0] \to M$ with $\gamma_n(0) = x$ such that
\begin{equation}\label{T1}
T^1_{\tau + n} \varphi(x) = \varphi(\gamma_n(-\tau - n)) + \int_{- \tau - n}^{0} L_1(s + \tau + n, \gamma_n(s), \dot\gamma_n(s))\;ds 
\end{equation}
and a minimal curve $\overline\gamma_n: [-\tau-n, 0] \to M$ with $\overline\gamma_n(0) = x$ such that
\begin{equation}\label{T2}
\overline{T}_{\tau + n} \varphi(x) = \varphi(\overline\gamma_n(-\tau - n)) + \int_{- \tau - n}^{0} \overline{L}(s + \tau + n, \overline\gamma_n(s), \dot{\overline\gamma}_n(s))\;ds. 
\end{equation}
Set $v_n= \dot\gamma_n(0)$ and $\overline{v}_{n} = \dot{\overline\gamma}_n(0)$. By the boundedness of velocities of minimizing curve  we have that $\{\gamma_n\}_{n \in \mathbf{N}}$ and $\{\overline\gamma_n\}_{n \in  \mathbf{N}}$ converge uniformly to $\gamma_{*}$ and $\overline\gamma_{*}$ on any closed interval of $(-\infty, 0]$. 
On the other hand, by \eqref{T1} for any $0 < t < n$ we have 
\begin{multline*}
T^1_{\tau+n}\varphi(x) = \varphi(\gamma_n(-\tau - n)) + \int_{-\tau-n}^{-t - \tau} L_1(s + \tau + n, \gamma_n(s), \dot\gamma_n(s))\;ds  
\\ + \int_{-t - \tau}^{0} L_1(s + \tau + n, \gamma_n(s), \dot\gamma_n(s))\;ds 
\\
\geq T^1_{n-t}\varphi(\gamma_n(-t - \tau)) + \int_{-t - \tau}^{0}L_1(s + \tau + n, \gamma_n(s), \dot\gamma_n(s))\;ds . 
\end{multline*}
Thus, we get 
\begin{multline*}
T^1_{\tau+n}\varphi(x)  - \min_{x \in M} T^1_{\tau+n}\varphi(x)  - T^1_{n-t}\varphi(\gamma_n(-t - \tau)) + \min_{x \in M} T^1_{n-t}\varphi(x)
\\
\geq \int_{-t - \tau}^{0} L_1(s + \tau + n, \gamma_n(s), \dot\gamma_n(s))\;ds + \min_{x \in M} T^1_{n-t}\varphi(x) -  \min_{x \in M} T^1_{\tau+n}\varphi(x).
\end{multline*}
Hence, as $n \uparrow \infty$ by \eqref{renorm} we deduce
\begin{equation}\label{L1}
w(x, [ \tau ]) - w(\gamma_{*}(-t - \tau), \langle  t\rangle) \geq \int_{-t- \tau}^{0} \overline{L}(s + \tau, \gamma_{*}(s), \dot\gamma_{*}(s))\;ds \mbox{ for any } t \in (-\infty, [ \tau ]]
\end{equation}
which implies, by a re-parametrization of the curves $\gamma_n$, that $\gamma_{*}$ is a calibrated curve for $w$ by definition \eqref{calibrated}. Similarly, from \eqref{T2} we obtain 
\begin{equation*}
\overline{T}_{\tau + n} \varphi(x) - \overline{T}_{n-t} \varphi(\overline\gamma_n(-t-\tau)) \geq \int_{-t-\tau}^{0} \overline{L}(s+\tau, \overline\gamma_n(s), \dot{\overline\gamma}(s))\;ds
\end{equation*}
which yields 
\begin{equation}\label{L2}
\overline{u}(x, \tau) - \overline{u}(\overline\gamma_{*}(-t-\tau), \langle t \rangle) \geq \int_{-t-\tau}^{0} \overline{L}(s+\tau, \overline\gamma_{*}(s), \dot{\overline\gamma}_{*}(s))\;ds \mbox{ for any } t \in (-\infty, [ \tau ]].
\end{equation}
Hence, $\overline\gamma_{*}$ is a calibrated curve for $\overline u$ and by \eqref{u_inf} it is also a calibrated curve for $w$ which contradicts the differentiability of $w$ in $x$. \qed

%where we also used the fact that $\langle t \rangle = \tau$. 

\medskip

\noindent {\it Proof of ($ii$) in \Cref{thm_2_intro}}. First, since 
\begin{equation*}
\lim_{n \to \infty} \| L_1(t+n, x, v) - \overline{L} (t, x, v) \|_{C^2(\R \times TM)} = 0
\end{equation*}
we have that 
\begin{equation}\label{M1}
\lim_{n \to \infty} \left( \mathcal{L}^1(t+n, B) - \overline{\mathcal{L}}(t, B) \right) = 0
\end{equation}
by the continuity of the Legendre Transform for any Borel compact subset $B$ of $TM$. Moreover, by compactness of $M$ we can find a finite subset $\{x_1, \dots, x_m\}$ of $M$ such that for any $x \in M$ there exists $x_i$ with $x \in B_{\eps}(x_i)$ and $t_{\eps} > 0$ for which by \Cref{M_convergence} we have
\begin{equation}\label{M2}
\rho\left(M_{L_1, x, t}(\varphi), M_{\overline{L}, x_i, t}(\varphi) \right) \leq \eps, \quad \forall\; t > t_{\eps}.
\end{equation}
Thus, combining \eqref{M1} and \eqref{M2} we get 
\begin{equation*}
\lim_{n \to \infty} d_{H} \left(\mathcal{G}\left(d T^1_{n+t} \varphi\right), \mathcal{G}\left(d \overline{T}_{n+t} \varphi\right)  \right) = 0, \quad \forall\; t \in \R.  
\end{equation*}
Hence, in conclusion, by triangular inequality and \Cref{thm_1} we get the result. \qed

\section{Second-order coupled oscillators}\label{Kuramoto}

We consider a system of $N$ coupled oscillators described by the following equations
\begin{equation*}
\ddot{\theta}_i=\Omega_i+\sum_{j=1}^N a_{ij}\sin(\theta_j-\theta_i)\ ,
\end{equation*}
where $\theta_i$ and $\Omega_i$ are the phase and natural frequency of i-th oscillator, respectively.
The coefficients $a_{ij}$ represent the coupling between the j-th
oscillator and the i-th oscillator and are symmetric.

This model is a modified version of the second-order Kuramoto model, which is closely related to the so-called swing equation, a fundamental tool in the analysis of power grid dynamics. 
Power grids are naturally modeled as networks of non-uniform, coupled oscillators with inertia, making the second-order Kuramoto model particularly suitable for capturing their behavior. By neglecting the first-order damping term present in standard formulations, our simplified version satisfies the Tonelli conditions for the Lagrangian, enabling a weak-KAM analysis of the system's dynamics. Moreover, we also take into account a generalized symmetric coupling $a_{ij}(t)$ which is a continuous and bounded function of time $t\in\mathbb{R}$, converging to a periodic function for $t\rightarrow\infty$.
For example, for a fixed $k\in\{1,\ldots,N\}$, we can consider the case 
\begin{equation*}
a_{ij}(t)=
    \begin{cases}
 e^{-\gamma t}\beta_{ij}(\omega t) & i=k \;\; \text{or} \;\; j=k,
 \vspace{.1cm}
 \\
 \beta_{ij}(\omega t) & i,j\neq k,
\end{cases}
\end{equation*}
where $\gamma>0$, $\omega\in\mathbb{R}^m$ ($m\leq N$) is a vector of rationally dependent frequencies and $\beta_{ij}(\omega t)$ is a bounded time-periodic function with period $T$.
Considering the equations of motion 
\begin{equation*}
    \ddot{\theta}_i=\Omega_i+\sum_{j=1}^N a_{ij}(t)\sin(\theta_j-\theta_i)\ ,
\end{equation*}
the potential can be written as
\begin{equation*}
   V(\theta,t)= -\langle \Omega,\theta\rangle-\frac{1}{2}\sum_{i,j=1}^N a_{ij}(t)\cos(\theta_j-\theta_i)\ .
\end{equation*}
Thus, the associated Lagrangian reads
\begin{equation}\label{Lagrangian}
    L(\theta,\dot\theta,t)=\frac{1}{2}\|\dot\theta\|^2+\langle \Omega,\theta\rangle+\frac{1}{2}\sum_{i,j=1}^N a_{ij}(t)\cos(\theta_j-\theta_i)
\end{equation}
and satisfies the Assumption 
\ref{TonelliAssumption}. Indeed, it is convex, superlinear w.r.t. $\dot\theta$ and the Euler-Lagrange flow solution of 
\begin{equation*}
    \ddot\theta_i = \Omega_i + \frac{1}{2}\sum_{j=1}^N a_{ij}(t)\cos(\theta_j-\theta_i)
\end{equation*}
is complete since the right-hand side is globally Lipschitz continuous w.r.t. the variable $\theta$. Moreover, the non-autonomous Lagrangian \eqref{Lagrangian} converges to
\begin{equation*}
    \overline{L_1}(\theta,\dot\theta,t)=\frac{1}{2}\|\dot\theta\|^2+\langle \Omega,\theta\rangle+\frac{1}{2}\sum_{\substack{i,j=1 \\ \scriptscriptstyle i,j\neq k}}^N a_{ij}(t)\cos(\theta_j-\theta_i)
\end{equation*}
in the sense of  \Cref{rate_Lagrangian}.
Note that, unlike \Cref{rate_Lagrangian}, here we are considering a generic period, but the previous analysis can be carried out in an analogous manner.
Next, we verify  such a condition: for a compact set $K\subset \mathbb{T}^N\times\mathbb{R}$, exploiting the symmetry of $a_{ij}(t)$ and the fact that the functions $\beta_{ij}(\omega t)$ being $T$-periodic functions we obtain 
\begin{equation*}
\begin{aligned}
    \| L&(\theta,\dot\theta,t+nT)-\overline{L_1}(\theta,\dot\theta,t)\|_{C^2(K)}=\bigg\|\sum_{i=1}^N a_{ik}(t)\cos(\theta_k-\theta_i)\bigg\|_{C^2(K)}\\
    &=\bigg\|\sum_{i=1}^N e^{-\gamma(t+nT)}\beta_{ik}(\omega t)\cos(\theta_k-\theta_i)\bigg\|_{C^2(K)}\\
    &=\sum_{|\alpha|\leq 2} \sup_{(\theta,t)\in K}\bigg|D^\alpha\bigg( \sum_{i=1}^N e^{-\gamma(t+nT)}\beta_{ik}(\omega t)\cos(\theta_k-\theta_i)\bigg)\bigg|\\
    &\leq \sum_{\alpha\leq 2} \sup_{t\in K}\bigg|\partial^\alpha_t\bigg( \sum_{i=1}^N e^{-\gamma(t+nT)}\beta_{ik}(\omega t)\bigg)\bigg|\\
    &\leq  \sup_{t\in K}\sum_{i=1}^Ne^{-\gamma(t+nT)}|\beta_{ik}(\omega t)|+\sup_{t\in K} \sum_{i=1}^N\left|\omega e^{-\gamma(t+nT)}\beta^\prime_{ik}(\omega t)-\gamma e^{-\gamma(t+nT)}\beta_{ik}(\omega t)\right|\\
    & \quad +\sup_{t\in K}\sum_{i=1}^N \bigg|\omega^2 e^{-\gamma(t+nT)}\beta^{\prime\prime}_{ik}(\omega t)-2\gamma\omega e^{-\gamma(t+nT)}\beta^\prime_{ik}(\omega t) +\gamma^2 e^{-\gamma(t+nT)}\beta_{ik}(\omega t)\bigg|\\
    &\leq e^{-\gamma nT} \sup_{t\in K}\sum_{i=1}^Ne^{-\gamma t}|\beta_{ik}(\omega t)|+e^{-\gamma nT}\sup_{t\in K} \sum_{i=1}^N\left(\omega e^{-\gamma t}|\beta^\prime_{ik}(\omega t)|+\gamma e^{-\gamma t}|\beta_{ik}(\omega t)|\right)\\
    & \quad +e^{-\gamma nT}\sup_{t\in K}\sum_{i=1}^N \bigg(\omega^2 e^{-\gamma t}|\beta^{\prime\prime}_{ik}(\omega t)|+2\gamma\omega e^{-\gamma t}|\beta^\prime_{ik}(\omega t)| +\gamma^2 e^{-\gamma t}|\beta_{ik}(\omega t)|\bigg)\\
    &\leq C e^{-\gamma nT}\ , \qquad \text{with} \; C\in \mathbb{R}\ ,
    \end{aligned}
\end{equation*}
where $\alpha$ (at the third line) denotes a multi-index and we used the boundedness of $\beta_{ik}$ and its derivative. Finally, in order to apply \Cref{thm_2_intro} we assume that the Aubry set associated with the limit time-periodic  Lagrangian consists of a unique hyperbolic periodic orbit. For instance, this happens when the Euler-Lagrange flow has a unique $T$-periodic solution and the associated eigenvalues of the monodromy matrix are all strictly smaller than 1.

We can now state the main result corresponding to \Cref{thm_2_intro} deducing the existence of an invariant torus for such a system.  

\begin{theorem}\label{thm_3}
    %Let $z:\mathbb{T}^N\times [0,\infty)\rightarrow\mathbb{R}$ be a solution of \eqref{Cauchy} for
    There exists a weak KAM invariant torus. Specifically, the following holds.
    \begin{itemize}
        \item[($i$)] Let $\varphi \in C(M; \R)$. Then, there exists a periodic weak KAM solution $w$ of the time-periodic system such that
        \begin{equation*}
            \lim_{n\to\infty} \| \mathcal{T}_{t+nT} \varphi(\theta)-w(\theta,[t])\|_{\infty}=0 \ ,
        \end{equation*}
        for $[t]=t$ mod $T$ and uniformly for $\theta\in\mathbb{T}^N$, where
        \begin{equation*}
            w(\theta,[t])=\overline{u}(\theta,[t])-\inf_{\theta\in\mathbb{T}^N}\overline{u}(\theta,[t])\ .
        \end{equation*}
        Moreover, uniformly w.r.t. the initial condition $z_0$ we have
        \begin{equation*}
            \| \mathcal{T}_{t+nT} \varphi(\theta)-w(\theta,[t])\|_{\infty}\leq C e^{-\gamma nT} \ , \quad \forall\, n\in \mathbb{N}\ ;
        \end{equation*}
        \item[($ii$)] Let $\varphi \in C(M; \R)$. Then, there exists an invariant torus given by the adherence of the graph of $d_\theta w$, in particular, 
        \[
        \displaystyle \lim_{n\to\infty}d_H(\overline{G}_n(d\mathcal{T}\varphi),\overline{G}(dw))=0,
        \]
        where
        \begin{equation*}
            G_n(d\mathcal{T}\varphi):=\Big\{\big(\theta,[t],d_\theta \mathcal{T}_{[t]+nT} \varphi(\theta),d_{t}\mathcal{T}_{[t]+nT} \varphi(\theta)\big): (\theta, [t]+nT) \in \mbox{Dom}(d\mathcal{T}\varphi) \Big\}
        \end{equation*}
        and $\overline{G}(dw)$ is the adherence of the periodic function $w$.
    \end{itemize}
\end{theorem}

\end{document}